\documentclass[12pt]{amsart}

\usepackage[utf8]{inputenc}
\usepackage[T1]{fontenc}
\usepackage{amsthm,amsmath,amssymb}
\usepackage{hyperref}
\usepackage{cleveref}
\usepackage[all]{xy}
 \newtheorem{Thm}{Theorem}[section]
 
 \newtheorem{Prop}[Thm]{Proposition}
 \newtheorem{Cor}[Thm]{Corollary}
\theoremstyle{remark}
 
\theoremstyle{definition}

\newcommand\ol{\overline}

\newcommand{\id}{\operatorname{id}}
\newcommand{\Id}{\operatorname{Id}}

\newcommand\inv{^{-1}}
\newcommand\leer{\text{---}}

\newcommand\Hom{\operatorname{Hom}}
\renewcommand\hom{\underline{\operatorname{hom}}}
\newcommand\homl{\hom_\ell}
\newcommand\hev{\epsilon}
\newcommand\hcoev{\eta}
\newcommand{\ot}{\otimes}
\newcommand\tR{\times_R}

\newcommand{\HMod}[4]{{{}^{#1}_{#3}}{\mathfrak M}^{#2}_{#4}}
\newcommand{\LMod}[1]{{}_{#1}\mathfrak M}

\newcommand{\LComod}[1]{\HMod{#1}{}{}{}}

\newcommand\WLC{\mathcal W_\ell}
\newcommand\LC{\mathcal Z_\ell}
\newcommand\WRC{\mathcal W_r}

\newcommand\CTR{\mathcal Z}
\newcommand\monstruk{\xi}

\newcommand\hbr{\sigma}

\newcommand\lidu[1]{#1^*}
\newcommand\Lidu[1]{#1{^{\vee}}}
\newcommand\ev{\operatorname{ev}}
\newcommand\db{\operatorname{db}}
\newcommand\op{{\operatorname{op}}}

\newcommand\C{\mathcal C}
\newcommand\D{\mathcal D}
\newcommand\F{\mathcal F}

\begin{document}
\title[Dual Hopf algebroids]{The dual and the double of a Hopf algebroid are Hopf algebroids}
\author{Peter Schauenburg}
\address{Institut de Math{\'e}matiques de Bourgogne --- UMR 5584 du CNRS\\
Universit{\'e} de Bourgogne\\
BP 47870, 21078 Dijon Cedex\\
France
}
\email{peter.schauenburg@u-bourgogne.fr}
\subjclass[2010]{16T99,18D10}
\keywords{bialgebroid, Hopf algebroid, duality}
\begin{abstract}
  Let $H$ be a $\times$-bialgebra in the sense of Takeuchi. We show that if $H$ is $\times$-Hopf, and if $H$ fulfills the finiteness condition necessary to define its skew dual $\Lidu H$, then the coopposite of the latter is $\times$-Hopf as well.

  If in addition the coopposite $\times$-bialgebra of $H$ is $\times$-Hopf, then the coopposite of the Drinfeld double of $H$ is $\times$-Hopf, as is the Drinfeld double itself, under an additional finiteness condition.
\end{abstract}
\maketitle
\section{Introduction}
Takeuchi \cite{Tak:GAAA} introduced $\times_R$-bialgebras as a generalization of ordinary bialgebras replacing the commutative base ring $k$ by a base algebra $R$ that can be noncommutative. Takeuchi's $\times$-algebras can be viewed as being to groupoids what bialgebras are to groups. We will also call them bialgebroids and refer to the survey \cite{Boh:HA} for references and relations to similar notions of Hopf algebroid or quantum groupoid.

In  \cite{MR1800718} three elements were added to $\times_R$-bialgebra theory. The first is a certain notion of $\times_R$-Hopf algebra. It is designed to be analogous in some respects to the notion of a Hopf algebra. The other two elements are the construction of the dual of a $\times_R$-bialgebra $H$, defined if a certain one of the four $R$-module structures of $H$ makes it finitely generated projective, and the construction of the Drinfeld double of $H$, defined if $H$ admits a dual and is Hopf.

These notions raise the following obvious questions:
\begin{itemize}
\item If $H$ is Hopf and admits a dual, is the latter also Hopf?
\item If $H$ is Hopf and admits a dual, is the double of $H$ also Hopf?
\end{itemize}
These questions are mentioned in \cite{Boh:HA}; it seems that they did not undergo a serious attack since, although we do believe that they should be considered rather basic for the theory of Hopf algebroids. A directly positive answer is perhaps to naive (see below) but we do give positive answers to what we believe are the ``correct'' versions of the questions.

$\times_R$-Hopf algebras can be viewed as defined by the existence of a certain additional structure map, which is not directly analogous to the antipode $S$ of an ordinary Hopf algebra, but to a combination of the antipode with the comultiplication. They can also be viewed as defined by a categorical property of their module categories. We use the latter characterization for our proofs; an explicit formula for the Hopf structure will not be given, simply because we did not manage to extract an explicit formula from the abstract arguments. This failure may be related to the fact that our result was so far unknown; it may equally well reflect the author's clumsiness, and perhaps someone else will write down the Hopf structures with ease.

\section{Dual and double Hopf algebroids}
\label{sec:dual-hopf-algebroids}

In this section we present and prove ---in a sense--- our main results on the Hopf properties of duals and doubles of Hopf algebroids. The proofs will be by reference to the next section, where the needed categorical results are proved. We also postpone until then recalling the necessary terminology, which we will assume known for the time being.

As to $\times_R$-bialgebras, we can afford to be succinct on the preliminaries, since the technical details of the definition of a $\times_R$-bialgebra are not needed explicitly.

Let $k$ be a commutative ring, and $R$ a $k$-algebra. Write $\ol R=R^\op$ for the opposite algebra of $R$, and $R^e=R\ot\ol R$ for the enveloping algebra.
An $R^e$-ring is a $k$-algebra $H$ equipped with an algebra homomorphism $R^e\to H$; in particular there is an underlying functor $\LMod H\to\LMod{R^e}$, and the latter is a monoidal category since it can be identified with the category of $R$-bimodules. A $\times_R$-bialgebra $H$ is an $R^e$-ring whose module category $\LMod H$ is endowed with the structure of a monoidal category such that the underlying functor $\LMod H\to \LMod{R^e}$ is a strict monoidal functor. This means that $H$ has a comultiplication that allows to endow the tensor product over $R$ of two left $H$-modules $M$ and $N$, taken with respect to the left $R$-module structure of $N$ and the left $\ol R$-module structure of $M$, with a left $H$-module structure.

By definition in \cite{MR1800718}, a $\times_R$-bialgebra $H$ is a $\times_R$-Hopf algebra if the underlying functor $\LMod H\to\LMod{R^e}$ preserves right inner hom-functors. The coopposite of $H$, which is a $\times_{R^\op}$-bialgebra, is $\times_{R^\op}$-Hopf if and only if the underlying functor preserves left inner hom-functors, since its module category is the same category with the order of tensor products reversed. We will call a $\times_R$-bialgebra whose coopposite is $\times_{R^\op}$-Hopf a $\times_R$-anti-Hopf algebra (following a usage of ``anti-Hopf'' introduced for ordinary bialgebras by Doi and Takeuchi).

A $\tR$-bialgebra $H$ is said to be left finite if it is finitely generated projective as left $\ol R$-module, and right finite if it is finitely generated projective as left $R$-module. If $H$ is left finite, then one can construct a $\times_R$-bialgebra structure on $\Lidu H:=\Hom_{\ol R-}(H,R)$, which we will call the left skew dual of $H$. We shun details on the $\times_R$-bialgebra structure in favor of a characterization of $\Lidu H$-modules assembled from \cite{MR1800718} and \cite{selfduality}. By the first, the skew dual $\Lidu H$ can be characterized through an equivalence of monoidal categories $\LComod H\cong\LMod{H^\vee}$. We do not recall the definition of $H$-comodules, but rather the fact that by the second reference, $\LComod H$ can in turn be characterized by an equivalence of $\LComod H$ with the left weak center $\WLC(\LMod H\to\LMod{R^e})$, whose definition we will recall in the next section. Thus, we have a category equivalence $\LMod{\Lidu H}\cong\WLC(\LMod H\to\LMod{R^e})$ commuting with the underlying functors to $\LMod{R^e}$.

After these more than flimsy explanations, we are ready to state:
\begin{Thm}
  Let $H$ be a left finite $\times_R$-Hopf algebroid. Then $\Lidu H$ is a $\times_R$-anti-Hopf algebroid.
\end{Thm}
\begin{proof}
  Clearly any object of $\LMod H$ is a direct limit of quotients of finitely generated projective objects. As observed in \cite{selfduality}, finitely generated projective $H$-modules admit left duals. Therefore, \cref{cor:1} shows that $\WLC(\LMod H\to\LMod{R^e})$ preserves left inner hom-functors. Thus $\LMod{\Lidu H}\to\LMod{R^e}$ preserves left inner hom-functors, whence the claim.
\end{proof}

\begin{Thm}
  Let $H$ be a left finite $\times_R$-Hopf algebra which is also a right finite $\times_R$-anti-Hopf algebra. Then $\Lidu H$ is $\times_R$-Hopf.
\end{Thm}
\begin{proof}
  Since $H$ is also anti-Hopf, the weak left center $\WLC(\LMod H\to\LMod{R^e})$ equals the right center. Thus, we can apply \cref{cor:1} to show that $\LMod H\to\LMod{R^e}$ preserves right inner hom-functors.
\end{proof}

In \cite{MR1800718} the Drinfeld double $D(H)$ of a left finite $\times_R$-Hopf algebra was defined to be a $\times_R$-bialgebra with the property that $\LMod{D(H)}\cong\WLC(\LMod H):=\WLC(\Id\colon\LMod H\to\LMod H)$. Without having to care about the details of the construction, we can obtain
\begin{Thm}
  Let $H$ be a left finte $\times_R$-Hopf algebra which is also $\times_R$-anti-Hopf. Then $D(H)$ is a $\times_R$-anti-Hopf algebra. If in addition $H$ is right finite, then $D(H)$ is also $\times_R$-Hopf.
\end{Thm}
\begin{proof}
  By the same reasoning as before, the underlying functor $\LMod{D(H)}\to\LMod H$ preserves left inner hom-functors. But so does $\LMod H\to\LMod{R^e}$ and thus $\LMod{D(H)}\to\LMod{R^e}$. Thus $D(H)$ is a $\times_R$-anti-Hopf algebra. Under the additional hypothesis, $\LMod{D(H)}\to\LMod H$ and therefore $\LMod{D(H)}\to\LMod{R^e}$ also preserves right inner-hom functors.
\end{proof}

\section{Hom functors in centralizer categories}
\label{sec:hom-funct-centr}

Our results on the duals and doubles of Hopf algebroids were based on tailor-made results on inner hom-functors in monoidal categories and generalizations of the Drinfeld center. This section contains those categorical arguments, following a presentation of the necessary categorical notions.

\subsection{Categorical preliminaries}
\label{sec:categ-prel}

We will use the language of monoidal categories, suppressing the associativity constraints in view of the well-known coherence results. A general reference is \cite{Kas:QG}.

We denote the neutral object of a monoidal category $\C$ by $I$. The left dual of an object $X\in\C$, if it exists, is denoted by $\lidu X$, it is equipped by definition with evaluation and coevaluation morphisms
\[\ev\colon \lidu X\ot X\to I\quad\text{and}\quad\db\colon I\to X\ot\lidu X.\]
An object admitting a left dual will be called left rigid.

A right inner hom-functor $\hom_r(X,\leer)$ in $\C$ is a right adjoint to the endofunctor $\leer\otimes X$, a left inner hom-functor $\homl(X,\leer)$ is a right adjoint to $X\ot\leer$. We denote the adjunction morphisms by $\hev\colon X\ot\homl(X,Y)\to Y$ and $\hcoev\colon Y\to\homl(X,X\ot Y)$.

If $X$ is left rigid, then $\leer\otimes X^*$ is a right inner hom-functor. If a right inner hom-functor $\hom_r(X,\leer)$ exists, and the canonical morphism $Y\ot\hom_r(X,\leer)\to \hom_r(X,Y)$ defined by the adjunction is an isomorphism for all $Y\in\C$, then $X$ is left rigid.

A monoidal functor $(\F,\monstruk)\colon \C\to \D$ between monoidal categories $\C$ and $\D$ consists of a functor, a natural isomorphism $\monstruk\colon \F(X\ot Y)\to \F(X)\ot \F(Y)$ (often suppressed) and an isomorphism $\monstruk_0\colon \F(I)\to I$ (very often suppressed) that satisfy coherence conditions with the associativity constraints of $\C$ and $\D$.

It is well-known that monoidal functors preserve dual objects: If $T\in\C$ has a left dual $\lidu T$, then $U=\F(T)$ has the left dual $\F(\lidu T)$ with evaluation $\F(\lidu T)\ot \F(T)\xrightarrow{\monstruk\inv} \F(\lidu T\ot T)\xrightarrow{\F(\ev)}\F(I)\cong I$ and coevaluation $I\cong \F(I)\xrightarrow{\F(\db)} \F(T\ot\lidu T)\xrightarrow{\monstruk} \F(T)\ot \F(\lidu T)$. We will take the liberty to identify the dual $\lidu U$ (normally defined already in the category $\D$) with $\F(\lidu T)$.

If right hom-functors $\homl(X,\leer)$ and $\homl(\F(X),\leer)$ exist, then there is a natural morphism $\F\homl(X,Y)\to\homl(\F(X),\F(Y))$ induced by $\F(X)\ot \F(\homl(X,Y))\xrightarrow{\monstruk} \F(X\ot\homl(X,Y))\xrightarrow{\F(\hev)}\F(Y)$. We say that $\F$ preserves left inner hom-functors, if these natural morphisms are isomorphisms.

  Let $(\F,\monstruk)\colon \C \to \D$ be a monoidal functor. The \emph{left
    weak centralizer} $\WLC(\F)$ of $\F$ is the category whose objects are pairs
  $(X,\hbr)$ in which $X$ is an object of $\D$, and
  $\hbr=\hbr_{X,V}\colon X\ot \F(V)\to \F(V)\ot X$ is a natural transformation
  making
  \begin{equation}
    \xymatrix{X\ot \F(V\ot W)\ar_{\hbr}[dd]\ar^-{X\ot\monstruk}[r]
&X\ot \F(V)\ot \F(W)\ar^{\hbr\ot \F(W)}[d]
\\
&\F(V)\ot X\ot \F(W)\ar^{\F(V)\ot\hbr}[d]
\\
\F(V\ot W)\ot X\ar_-{\monstruk\ot X}[r]
&\F(V)\ot \F(W)\ot X}
  \end{equation}
  commute and satisfying $\hbr=\id:X\ot \F(I)\to \F(I)\ot X$.

  The left weak centralizer is a monoidal category, the tensor product
  of $X$ and $Y$ being the tensor product $X\ot Y$ in $\D$ endowed with
  \begin{equation*}
    \hbr=\left( X\ot Y\ot \F(V)\overset{X\ot\hbr}\longrightarrow X\ot \F(V)\ot Y\overset{\hbr\ot Y}\longrightarrow \F(V)\ot X\ot Y\right)
  \end{equation*}
We refer to $\hbr$ as a half-braiding, though this is more customary for the special case that $\C=\D$ and $\F$ is the identity functor.
  
The right weak centralizer $\WRC(\F)$ is defined similarly. The (left)
centralizer $\CTR(\F)=\LC(\F)$ consists of those objects $(X,\hbr)$
in the left weak centralizer where $\hbr$ is an isomorphism. It is
naturally equivalent to the right centralizer. Centralizers and weak
centralizers were introduced by Majid \cite{Maj:RDQDMC}, who calls the
right weak centralizer of $\F$ the dual of the functored category
$(\C,\F)$.

We note that if $ V\in\C$ is left rigid, and $ X\in\WLC(\F) $, then $ \hbr_{X,\lidu V} $ is an isomorphism. The inverse
$\hbr_{X,\lidu V}\inv\in\D(\F(\lidu V)\ot X,X\ot \F(\lidu V))$ corresponds to $\hbr_{X,V}$ under the bijection
\begin{equation*}
  \D(\F(\lidu V)\ot X,X\ot \F(\lidu V))\cong \C(X\ot \F(V),\F(V)\ot X)
\end{equation*}
coming from $\leer\ot \F(\lidu V)$ being a right inner hom-functor and $\F(V)\ot\leer$ being a left inner hom-functor.

\subsection{Hom functors in centralizer categories}
\label{sec:hom-funct-centr-1}
Our key result says that under favorable conditions weak centralizer categories have inner hom-functors preserved by underlying functors:

\begin{Prop}\label{prop:1}
  Let $\F\colon\C\to\D$ be a monoidal functor. Assume that $\D$ has left inner hom-functors, and that $\C$ is left rigid. Then $\WLC(\F)$ admits left inner hom-functors, and the underlying functor $\WLC(\F)\to\D$ preserves them.
\end{Prop}
\begin{proof}
  In other words, for $X,Y\in\D$ and $T\in\C$ we have to construct a half-braiding
  \begin{equation}
    \label{eq:1}
    \hbr=\hbr_{\hom(X,Y),T}\colon\homl(X,Y)\ot \F(T)\to\F(T)\ot\homl(X,Y)
  \end{equation}
  and show that the unit and counit
  \begin{align}
    \label{eq:2}
    X\ot\homl(X,Y)&\to Y\\
    Y&\to\homl(X,X\ot Y)
  \end{align}  
  of the hom-tensor adjunction are morphisms in $\WLC(\F)$.

  Not the least challenge in the proof is fitting the necessary diagrams on a page. To manage this, we shall write $\F(T)=U$ and identify $\F(\lidu T)$ with $\lidu U$. Moreover, we will write $\homl(X,Y)=:[X,Y]=:[XY]$, and in commutative diagrams we shall moreover suppress the tensor product symbol.

  We use the bijection
  \begin{equation}
    \label{eq:4}
    \D([X,Y]\ot U,U\ot[X,Y])\cong\D(X\ot\lidu U\ot [X,Y]\ot U,Y)
  \end{equation}
  to define $\hbr_{[X,Y],\leer}$ by commutativity of
  \begin{equation*}
    \xymatrix{X\lidu U[X,Y]U\ar[rr]^{X\lidu U\hbr}\ar[d]_{\hbr[X,Y]U}&&X\lidu UU[X,Y]\ar[d]^{X\ev[X,Y]}\\
      \lidu U X[X,Y]U\ar[d]_{\lidu U\hev U}&&X[X,Y]\ar[d]^{\hev}\\
      \lidu U YU\ar[r]^{\lidu U\hbr}&\lidu UU Y\ar[r]^{\ev Y}&Y}
  \end{equation*}

  Given that $\hbr_{X,\F(\lidu T)}$ is an isomorphism, the same diagram expresses the image of
  \begin{equation*}
    X\ot\homl(X,Y)\ot U\xrightarrow{\hev\ot U}Y\ot U\xrightarrow\hbr U\ot Y
  \end{equation*}
  under the bijection
  \begin{equation*}
    \C(X[XY]U,UY,UY)\cong\C(\lidu UX[XY]U,Y).
  \end{equation*}
  Therefore, commutativity of
  \begin{equation*}
    \xymatrix{\lidu UX[XY]U\ar[r]^{\lidu UX\hbr}
      &\lidu UXU[XY]\ar[r]^{\lidu U\hbr[XY]}&\lidu UUX[XY]\ar[ddl]^{\ev X[XY]}\ar[d]^{\lidu UU\hev}\\
      X\lidu U[XY]U\ar[r]^{X\lidu U\hbr}\ar[u]^{\hbr[XY]U}&X\lidu UU[XY]\ar[u]^{\hbr U[XY]}\ar[d]_{X\ev[XY]}&\lidu UUY\ar[d]^{\ev Y}\\
      &X[XY]\ar[r]_{\hev}&Y
    }
  \end{equation*}
  proves that
  \begin{equation*}
    \xymatrix{X[XY]U\ar[r]^{X\hbr}\ar[d]^{\hev U}&XU[XY]\ar[r]^{\hbr[XY]}&UX[XY]\ar[d]_{U\hev}\\
      YU\ar[rr]^{\hbr}&&UY}
  \end{equation*}
  commutes.

  For $S,T\in\C$, $U=\F(S)$ and $V=\F(T)$, the diagram
  {\small
  \begin{equation*}
    \xymatrix{X\lidu V\lidu U[X,Y]UV\ar@/^/[r]^-{X\lidu V\lidu U\hbr V}\ar[d]_{\hbr\lidu U[XY]UV}
      &X\lidu V\lidu UU[X,Y]V\ar@/^/[r]^-{X\lidu V\lidu UU\hbr}\ar[d]_{\hbr\lidu UU[XY]V}\ar[ddr]^(.8){X\lidu V\ev[XY]V}
      &X\lidu V\lidu UUV[X,Y]\ar@<2em>@(r,r)_(.3){X\lidu V\ev V[XY]}[ddd]\\
      \lidu VX\lidu U[X,Y]UV\ar@/^/[r]^-{\lidu VX\lidu U\hbr V}\ar[d]_{\lidu V\hbr[XY]V}
      &\lidu VX\lidu UU[X,Y]V\ar[d]_{\lidu VX\ev[XY]V}&
      \\
      \lidu V\lidu UX[X,Y]UV\ar[d]^{\lidu V\lidu U\hev UV}
      &\lidu VX[X,Y]V\ar[d]^{\lidu V\hev V}&
      X\lidu V[X,Y]V\ar[d]^{X\lidu V\hbr}\ar[l]^{\hbr[XY]V}\\
      \lidu V\lidu UYUV\ar[d]^{\lidu V\lidu U\hbr V}&\lidu VYV\ar[d]^{\lidu V\hbr}&X\lidu VV[X,Y]\ar[d]^{X\ev[XY]}\\
      \lidu V\lidu UUYV\ar[ur]_{\lidu V\ev YV}\ar[d]^{\lidu V\lidu UU\hbr}&\lidu VVY\ar[dr]^{\ev Y}&X[X,Y]\ar[d]^{\hev}\\
      \lidu V\lidu UUVY\ar[ur]_{\lidu V\ev VY}&&Y}
  \end{equation*}}
  shows that the latter is a half-braiding. The two heptagons in the diagram are the definition of $\hbr_{\homl(X,Y),\leer}$, while the remaining fields commute for trivial reasons
  
  That the unit $\hcoev\colon Y\to\homl(X,X\ot Y)$ of the hom-tensor adjunction is a morphism in $\WLC(\F)$ amounts to equality of two morphisms $Y\ot \F(T)\to \F(T)\ot\homl(X,X\ot Y)$ which we will verify after applying the bijection
  \begin{equation*}
    \D(Y\ot U,U\ot[X,X\ot Y])\cong\D(X\ot \lidu U\ot Y\ot U, X\ot Y)
  \end{equation*}
  by the diagrams
  \begin{equation*}
    \xymatrix{X\lidu UYU\ar@<1ex>[r]^{X\lidu U\hcoev U}\ar[d]_{\hbr YU}
      &X\lidu U[X,XY]U\ar[r]^{X\lidu U\hbr}\ar[d]^{\hbr[X,XY]U}
      &X\lidu UU[X,XY]\ar[d]^{X\ev[X,XY]}\\
      \lidu UXYU\ar@<1ex>[r]^{\lidu UX\hcoev U}\ar@{=}[d]
      &\lidu UX[X,XY]U\ar[dl]_{\lidu U\ev U}&X[X,XY]\ar[d]^{\hev}\\
      \lidu UXYU\ar[r]^{\lidu U\hbr}&\lidu UUXY\ar[r]^{\ev XY}&XY}
  \end{equation*}
  and
  \begin{equation*}
    \xymatrix{X\lidu UYU\ar[r]^{X\lidu U\hbr}\ar[d]_{\hbr YU}
      &X\lidu UUY\ar[ddl]^{\hbr UY}\ar[d]^{X\ev Y}\ar@<1ex>[r]^{X\lidu UU\hcoev}
    &X\lidu UU[X,XY]\ar[d]^{X\ev[X,XY]}\\
    \lidu UXYU\ar[d]_{\lidu UX\hbr}&XY\ar[r]^{X\hcoev Y}\ar@{=}[dr]&X[X,XY]\ar[d]^{\ev}\\
    \lidu UXUY\ar[r]_{\lidu U\hbr Y}&\lidu UUXY\ar[u]^{\ev XY}&XY}
  \end{equation*}
\end{proof}

\begin{Cor}\label{cor:1}
  Let $\C$ and $\D$ be abelian monoidal categories with colimits, and $\F\colon \C\to \D$ a cocontinuous additive monoidal functor.

  Assume that $\D$ has left inner hom-functors. Assume that every object in $\C$ is the directed union of finitely generated subobjects, that every finitely generated object of $\C$ is the quotient of a finitely generated projective object, and that finitely generated projective objects in $\C$ admit left duals.

  Then $\WLC(\F)$ has left inner hom-functors, and the underlying functor $\WLC(\F)\to\D$ preserves them.
\end{Cor}
\begin{proof}
  Arguing as in the proof of the preceding proposition, we can first define $\hbr\colon\homl(X,Y)\ot\F(P)\to \F(P)\ot\homl(X,Y)$ for finitely generated projective $P$. Any finitely generated $M\in\C$ can be written as the colimit of a diagram $P_i\xrightarrow{f_i}P $ by choosing an epimorphism $P\to M$ from a finitely generated projective $P$, and writing the kernel as the directed union of finitely generated subobjects, which are epimorphic images of finitely generated projective objects $P_i$. By our hypotheses, this allows to extend the definition of the half-braiding to $\hbr\colon\homl(X,Y)\ot\F(M)\to\F(M)\ot\homl(X,Y)$; this in a unique and natural way, since every morphism $M\to M'$ can be lifted, for any choice of epimorphisms $P\to M$ and $P'\to M'$, to a commutative square by a morphism $P\to P'$. Finally, the half-braiding can be extended to any object of $\C$ since these objects are assumed to be directed unions of finitely generated subobjects.
\end{proof}

\bibliographystyle{plain}
\bibliography{eigene,andere,arxiv,mathscinet}
\end{document}